\documentclass[final,times]{elsarticle}
\usepackage{epsfig}
\usepackage{amssymb}
\usepackage{amsmath}
\usepackage{amsthm}
\usepackage{mathtools}
\usepackage{enumitem}
\usepackage{caption}
\usepackage{array,booktabs}
\usepackage[hyphens]{url}
\usepackage{psfrag,subfigure}
\usepackage[hidelinks]{hyperref}
\hypersetup{breaklinks=true}
\usepackage[english]{babel}
\usepackage{amsfonts, verbatim}
\usepackage{amsmath}
\usepackage{color,tikz}  

\newcommand{\R}{{\mathbb{R}}}
\newcommand{\Vfree}{V}

\newcommand{\C}[1]{\mathbf{C^{#1}}}

\newtheorem{lem}{Lemma}[section]
\newtheorem{pro}{Proposition}[section]

\newtheorem{defi}{Definition}[section]
\theoremstyle{definition}
\newtheorem{example}{Example}[section]
 \renewcommand{\L}[1]{\mathbf{L^{#1}}}
\newcommand{\Lloc}[1]{\mathbf{L^{#1}_{loc}}}

\newcommand{\tv}{\mathrm{TV}}
\newcommand{\ronec}{\mathcal{R}_1^{\rm c}}
\newcommand{\rtwoc}{\mathcal{R}^{\rm c}_2}
\makeatletter
\newcommand{\oset}[2]{%
  {\mathop{#2}\limits^{\vbox to -.5\ex@{\kern-\tw@\ex@
   \hbox{\scriptsize #1}\vss}}}}
\makeatother

\journal{}

\usepackage{calc}
\usepackage{lipsum}
\makeatletter
\def\ps@pprintTitle{%
 \let\@oddhead\@empty
 \let\@evenhead\@empty
 \def\@oddfoot{}%
 \let\@evenfoot\@oddfoot}
\makeatother

\begin{document}

\begin{frontmatter}

%% Title, authors and addresses

%% use the tnoteref command within \title for footnotes;
%% use the tnotetext command for theassociated footnote;
%% use the fnref command within \author or \address for footnotes;
%% use the fntext command for theassociated footnote;
%% use the corref command within \author for corresponding author footnotes;
%% use the cortext command for theassociated footnote;
%% use the ead command for the email address,
%% and the form \ead[url] for the home page:
%% \title{Title\tnoteref{label1}}
%% \tnotetext[label1]{}
%% \author{Name\corref{cor1}\fnref{label2}}
%% \ead{email address}
%% \ead[url]{home page}
%% \fntext[label2]{}
%% \cortext[cor1]{}
%% \address{Address\fnref{label3}}
%% \fntext[label3]{}

\title{A macroscopic traffic model with phase transitions and local point constraints on the flow}
%\title{Numerical simulations for conservation laws with non-local point constraints in crowd dynamics}

%% use optional labels to link authors explicitly to addresses:
%% \author[label1,label2]{}
%% \address[label1]{}
%% \address[label2]{}

\author[L'Aquila]{Benyahia Mohamed\corref{mycorrespondingauthor}}
\cortext[mycorrespondingauthor]{Corresponding author}
\ead{benyahia.ramiz@gmail.com}

\author[Warsaw]{Massimiliano D.~Rosini}
\ead{mrosini@umcs.lublin.pl}

\address[L'Aquila]{Gran Sasso Science Institute\\
Viale F. Crispi 7, 67100 L'Aquila, Italy}

\address[Warsaw]{
Instytut Matematyki, Uniwersytet Marii Curie-Sk\l odowskiej\\
Plac Marii Curie-Sk\l odowskiej 1, 20-031 Lublin, Poland}

\begin{abstract}
In this paper we present a two phase model for vehicular traffics subject to point constraints on the flow, its motivation being, for instance, the modelling of the effects of toll booths along a road.
\end{abstract}


\end{frontmatter}

%% \linenumbers

%\tableofcontents

\section{Introduction}

Macroscopic traffic flow models has been a growing field of research in the last decade as it is finding real life applications related to traffic control, management and prediction, see the surveys \cite{bellomo2011modeling, survey2013, piccolitosinreview}, the books \cite{garavello2006traffic, Rosinibook} and the reference therein. 
Among these models, two of most noticeable importance are Lighthill, Whitham \cite{LWR1} and Richards \cite{LWR2} model (LWR)
\begin{align*}
&\rho_t+(v \, \rho)_x=0,
&v=\Vfree(\rho),
\end{align*}
and the Aw, Rascle \cite{ARZ1} and Zhang \cite{ARZ2} model (ARZ)
\begin{align*}
&\rho_t+(v \, \rho)_x=0,
&[\rho\,(v + p(\rho)]_t + [v \, \rho \, (v + p(\rho)]_x =0.
\end{align*}
Theses two models aim to predict the evolution in time $t$ of the density $\rho$ and the velocity $v$ of vehicles moving along a homogeneous highway parametrized by the coordinate $x \in \R$ and with no entries and exits.
Both of the models have drawbacks however in their modelling of the evolution of traffic flows. 
Indeed, LWR assigns a priori an explicit relation between density and velocity, that means that the model assumes that the velocity of the vehicles is entirely determined by the density and dismisses the possibility of different vehicle populations that might exhibit a different velocity for a given density.
Empirical studies show however that the density-flux diagram can be approximated by this kind of models only up to a certain density, across which the traffic changes phase from free to congested, the latter being better approximated by a second order models such as ARZ.
On the other hand ARZ is not well-posed near the vacuum.
In particular, when the density is close to zero, the solution does not depend in general continuously on the initial data.

This motivated the introduction in \cite{goatin2006aw} of a two phase model that describes free and congested phases by means of respectively LWR and ARZ.
Recall that this model was recently generalized in \cite{BenyahiaRosini01}. 
A couple of mathematical difficulties have to be highlighted.
First, the model consists of two systems of equations that prescribe the evolution in time of the traffic in two different phases, free and congested, and prescribes a specific set of admissible phase transitions. 
%This means that if initially the traffic in the road is entirely lying in the free flow phase or entirely lying in the congested phase, then its evolution shall be described by either one of the system of equations. 
%But if initially the traffic lies in the free flow phase in some region of the road and is in the congested phase in another region, then the model aim to provide for a solution that solves the system of the free flow phase in the region where it belongs in the free flow phase, similarly in the congested phase while the discontinuities between the two phases that this solution exhibits, must satisfy a priori rules. 
One difficulty is that it is not known a priori the curves in the $(x,t)$ plane dividing two phases.
The problem cannot be reduced therefore into solving two different systems in two distinct regions with prescribed boundary conditions.
As a consequence, defining a notion of weak or entropy solution via an integral condition (as it is standard in the field of PDEs) turns out to be a delicate task. 
In \cite{BenyahiaRosini01} it has been possible to do so when the characteristic field of the free phase is linearly degenerate, because in this case the flux in the free phase reduces to a restriction of the flux diagram of ARZ, which allowed to make use of the definition of weak and entropy solution introduced in \cite{BCM2order}.\newline
Another difficulty is the possibility that two phase transitions may interact with each other and cancel themselves.
In fact, for instance, it is perfectly reasonable to consider a traffic characterized by a single congested region $C$, with vehicles emerging at the front end of $C$ and moving into a free phase region with a velocity higher than the tail of the queue at the back end of $C$, so that after a certain time the congested region disappears and all the traffic is in a free phase. 
For this reason a global approach for the study of the corresponding Cauchy problem can not be applied, as it would require a priori knowledge of the phase transition curves; it is instead preferable to apply the wave-front tracking algorithm, as it allows to track the positions of the phase transitions.

The present article deals with the Riemann problem for the two phase model introduced in \cite{BenyahiaRosini01}, equipped with a local point constraint on the flow, meaning that we add the further condition that at the interface $x=0$ the flow of the solution must be lower than a given constant quantity $Q$. 
This models, for instance, the presence of a toll gate across which the flow of the vehicles cannot exceed the value $Q$.
The additional difficulty that this add to the mathematical modelling of the problem is that this time one can start with a traffic that is initially completely in the free phase, but congested phases arise in a finite time in the upstream of $x=0$, as it is perfectly reasonable in the case of a toll gate with a very limited capacity.
The aim of this paper is to establish two Riemann solvers for this model and to study their properties. 
These Riemann solvers will be the basis upon which to rely for the ulterior study of the Cauchy problem.

The paper is organized as follows.
In Section~\ref{sec:2} we state carefully the problem, introduce the needed notations and define the two Riemann solvers, see Definition~\ref{def:R1c} and Definition~\ref{def:R2c}. 
Then in Section~\ref{sec:3} we study their basic properties, some of which require lenghty and technical proofs that are postponed in the last section.

\section{The model and the main result}\label{sec:2}

The aim of this section is to propose two Riemann solvers for the two phase transition model developed in \cite{BenyahiaRosini01}, coupled with a point constraint on the flow \eqref{eq:model}.

\subsection{Assumptions and notations}%\label{sec:assandnot}

Let us first introduce the notation to the reader and explain its usage.
Fix $R_{\rm f}'' > R_{\rm f}' > 0$.
Consider the maps $\Vfree \in \C2([0,R_{\rm f}''];\R_+)$ and $p \in \C2([R_{\rm f}',+\infty);\R)$ such that $\Vfree(R_{\rm f}'') > 0$ and
\begin{gather}\tag{{\rm\bf H1}}\label{H1}
\begin{array}{c}
%\Vfree(R_{\rm f}'') > 0 \text{ and }
\Vfree_{\rho}(\rho) \leq 0,
~
\Vfree(\rho) + \rho \, \Vfree_{\rho}(\rho) > 0,
~
2\Vfree_{\rho}(\rho) + \rho \, \Vfree_{\rho\rho}(\rho) \leq 0
\quad\forall%\text{ for every }
\rho \in [0,R_{\rm f}''],
\end{array}
\\[5pt]\tag{{\rm\bf H2}}\label{H2}
\begin{array}{c}
p_{\rho}(\rho)>0,
~
2p_{\rho}(\rho) + \rho \, p_{\rho\rho}(\rho)>0
\quad\forall%\text{ for every }
\rho \in [R_{\rm f}',+\infty),
\end{array}
\\[5pt]\tag{{\rm\bf H3}}\label{H3}
\begin{array}{c}
\Vfree_{\rho}(\rho)+p_{\rho}(\rho)>0,
~
\Vfree(\rho) < \rho \, p_{\rho}(\rho) 
\quad\forall%\text{ for every }
\rho \in [R_{\rm f}',R_{\rm f}''].
\end{array}
\end{gather}
For later convenience, see \figurename~\ref{fig:domains}, we introduce the following notation:
\begin{align*}
&V_{\max} \doteq \Vfree(0),&
&V_{\min} \doteq \Vfree(R_{\rm f}''),
\\
&W_{\max} \doteq p(R_{\rm f}'') + V_{\min},&
&W_{\rm c} \doteq p(R_{\rm f}') + \Vfree(R_{\rm f}'),\\
&R_{\max} \doteq p^{-1}(W_{\max}),&
&R_{\rm c} \doteq p^{-1}(W_{\rm c}),&
&q_{\max} = R_{\rm f}'' \, V_{\min}.
\end{align*}
By definition we have $R_{\max}>R_{\rm f}''>0$, $R_{\rm c}>R_{\rm f}'>0$, $W_{\max}>W_{\rm c}$ and by \eqref{H2} the map $p^{-1} \colon [W_{\rm c} - \Vfree(R_{\rm f}'),W_{\max}] \to [R_{\rm f}',R_{\max}]$ is increasing.

To avoid technicalities in the exposition, we will throughout assume that $p$ is defined in $(0,+\infty)$ and that there it satisfies \eqref{H2}.

Some simple choices for $\Vfree$ and $p$, see \cite{AKMR2002, ARZ1, goatin2006aw}, are
\begin{align*}%\label{eq:ex1v}
\Vfree(\rho) \doteq v_{\max} \left[1-\frac{\rho}{R}\right],&&
p(\rho) \doteq
\begin{cases}
\dfrac{v_{\rm ref}}{\gamma} \left[\dfrac{\rho}{\rho_{\max}}\right]^\gamma,& \gamma>0,\\[7pt]
v_{\rm ref} \, \log\left[\dfrac{\rho}{\rho_{\max}}\right],& \gamma=0,
\end{cases}
\end{align*}
where $v_{\max}$, $R$, $\gamma$, $v_{\rm ref}$ and $\rho_{\max}$ are strictly positive parameters, that can be chosen so that \eqref{H1}-\eqref{H3} are satisfied, see \cite{BenyahiaRosini01} for the details.

\begin{figure}[ht]
\centering{
\begin{psfrags}
      \psfrag{a}[l,B]{$\rho$}
      \psfrag{b}[c,B]{$R_{\max}$}
      \psfrag{c}[c,B]{$~~R_{\rm f}''$}
      \psfrag{d}[c,B]{$R_{\rm c}~$}
      \psfrag{e}[c,B]{$R_{\rm f}'$}
      \psfrag{f}[r,B]{$R_{\rm f}'' \, V_{\min}$}
      \psfrag{g}[r,B]{$Q(u)$}
      \psfrag{h}[c,B]{$V_{\max}$}
      \psfrag{i}[c,B]{$V_{\min}$}
      \psfrag{l}[c,B]{$V_{\rm c}$}
      \psfrag{m}[c,c]{$\Omega_{\rm c}$}
      \psfrag{n}[c,B]{$\Omega_{\rm f}''~$}
      \psfrag{o}[c,B]{$\Omega_{\rm f}'$}
      \includegraphics[width=.3\textwidth]{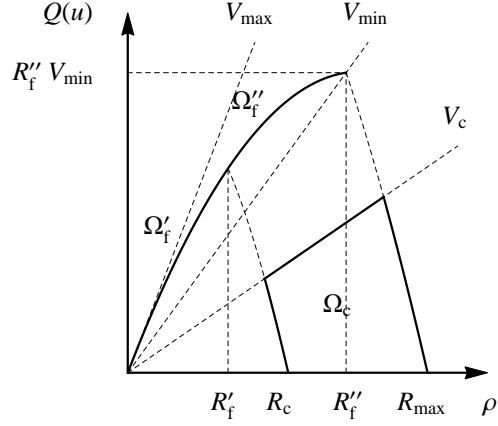}
\end{psfrags}}
\caption{Geometrical meaning of notations used through the paper.}
\label{fig:domains}
\end{figure}

Fix $V_{\rm c} \in \,]0, V_{\min}]$ and let $u \doteq (\rho,v)$ belong to $\Omega \doteq \Omega_{\rm f} \cup \Omega_{\rm c}$, where
\begin{align*}
&\Omega_{\rm f} \doteq \left\{u \in [0,R_{\rm f}''] \times [V_{\min},V_{\max}] \colon v = \Vfree(\rho)\right\}
&\text{and}&
&\Omega_{\rm c} \doteq \left\{\vphantom{R_{\rm f}''} u \in [R_{\rm f}',R_{\max}] \times [0,V_{\rm c}] \colon  W_{\rm c} \leq v + p(\rho) \leq W_{\max}\right\}
\end{align*}
are the domains of respectively free phases and congested phases.
Observe that $\Omega_{\rm f}$ and $\Omega_{\rm c}$ are invariant domains for respectively LWR and ARZ.
Introduce also the domains
\begin{align*}
&\Omega_{\rm f}' \doteq \left\{u \in \Omega_{\rm f} \colon \rho \in [0,R_{\rm f}') \right\},
&\Omega_{\rm f}'' \doteq \left\{u \in \Omega_{\rm f} \colon \rho \in [R_{\rm f}',R_{\rm f}''] \right\}.
\end{align*}
We finally introduce $u_{\rm c} \doteq (\rho_{\rm c} , V_{\rm c}) \colon \left[ W_{\rm c}, W_{\max}\right] \to \Omega_{\rm c}$ defined by $\rho_{\rm c}(w) \doteq p^{-1}(w-V_{\rm c})$, and $u_{\rm f} \doteq (\rho_{\rm f} , v_{\rm f}) \colon [ W_{\rm c}, W_{\max}] \to \Omega_{\rm f}''$ implicitly defined by $v_{\rm f}(w) = \Vfree(\rho_{\rm f}(w)) = w-p(\rho_{\rm f}(w))$.
We point out that $R_{\rm f}' = \rho_{\rm f}(W_{\rm c})$ and $R_{\rm f}'' = \rho_{\rm f}(W_{\max})$.

\subsection{The unconstrained Riemann problem}%\label{sec:PT}

The Riemann problem for the two phase model introduced in \cite{BenyahiaRosini01} has the form
\begin{align}\label{eq:model}
\begin{array}{@{}l@{}}
\text{\textbf{Free flow} (LWR)}\\[2pt]
\begin{cases}
u \in \Omega_{\rm f},\\
\rho_t+Q(u)_x=0,\\
v=\Vfree(\rho),
\end{cases}
\end{array}
&&
\begin{array}{l@{}}
\text{\textbf{Congested flow} (ARZ)}\\[2pt]
\begin{cases}
u \in \Omega_{\rm c},\\
\rho_t+Q(u)_x=0,\\
\left[\rho \, W(u)\right]_t + \left[Q(u) \, W(u)\right]_x=0,
\end{cases}
\end{array}
&&
\begin{array}{@{}l@{}}
\text{\textbf{Initial datum}}\\[2pt]
u(0,x) = \begin{cases}
u_\ell&\text{if }x<0,\\
u_r&\text{if }x>0,
\end{cases}\\\\
\end{array}
\end{align}
where $u_\ell,u_r \in \Omega$ and the maps $Q \colon \Omega \to [0,q_{\max}]$ and $W \colon \Omega \to [W_{\rm c},W_{\max}]$ are defined by
\begin{align*}
Q(u)\doteq \rho \, v,&
&W(u) \doteq 
\begin{cases}
v+p(\rho)&\text{if } u \in \Omega_{\rm c} \times \Omega_{\rm f}'',\\
W_{\rm c}&\text{if } u \in \Omega_{\rm f}'.
\end{cases}
\end{align*}
Above $\rho$ and $v$ denote respectively the density and the speed of the vehicles, while $\Vfree$ and $p$ give respectively the speed of the vehicles in a free flow and the ``pressure'' of the vehicles in a congested flow.
In the free phase the characteristic speed is $\lambda_{\rm f}(u) \doteq \Vfree(\rho) + \rho \, \Vfree_{\rho}(\rho)$.
In the following table we collect the informations on the system modelling the congested phase:
\begin{align*}
&r_1(u) \doteq (1,-p_{\rho}(\rho)),&
&r_2(u) \doteq (1,0),\\
%&r_1(u) \doteq \begin{pmatrix}1\\-p_{\rho}(\rho)\end{pmatrix},&
%&r_2(u) \doteq \begin{pmatrix}1\\0\end{pmatrix},\\
&\lambda_1(u) \doteq v-\rho\,p_{\rho}(\rho),&
&\lambda_2(u) \doteq v,\\
&\nabla\lambda_1\cdot r_1(u) = -2p_{\rho}(\rho)-\rho\,p_{\rho\rho}(\rho),&
&\nabla\lambda_2\cdot r_2(u) = 0,\\
&\mathcal{L}_1(\rho;u_0) \doteq W(u_0)-p(\rho),&
&\mathcal{L}_2(\rho;u_0) \doteq v_0.
%,\\
%&w_1(u) \doteq v,&
%&W(u) \doteq W(u).
\end{align*}
Above $r_i$ is the $i$-th right eigenvector, $\lambda_i$ is the corresponding eigenvalue and $\mathcal{L}_i$ is the $i$-th Lax curve. %and $w_i$ is the $i$-th Riemann invariant.
By the assumptions \eqref{H1} and \eqref{H2} the characteristic speeds are bounded by the velocity, $\lambda_{\rm f}(u) \le v$, $\lambda_1(u) \leq \lambda_2(u) = v$, and $\lambda_1$ is genuinely non-linear, $\nabla\lambda_1 \cdot r_1(u) \neq0$.
Beside the Riemann solver introduced in \cite{BenyahiaRosini01}, here denoted by $\mathcal{R}_1$, we introduce in the following definition also a second Riemann solver $\mathcal{R}_2$.
Roughly speaking, the motivation is that the solution corresponding to $\mathcal{R}_2$ has flow through the constraint higher than that corresponding to $\mathcal{R}_1$.
We denote by $\mathcal{R}_{\rm LWR}$ and $\mathcal{R}_{\rm ARZ}$ the Riemann solvers for respectively LWR and ARZ.
To simplify our notation, we let $q_* \doteq Q(u_*)$ and $w_* \doteq W(u_*)$.
Recall that for any $u_\ell,u_r \in \Omega$ with $\rho_\ell\neq\rho_r$, the speed of propagation of a discontinuity between $u_\ell$ and $u_r$ is $\sigma(u_\ell,u_r) \doteq [q_r-q_\ell]/[\rho_r-\rho_\ell]$.

\begin{defi}\label{def:LWR-ARZ}
The Riemann solver $\mathcal{R}_1 \colon \Omega^2 \to \L\infty(\R;\Omega)$ is defined as follows:

\begin{enumerate}[label={($R_1$\alph*)},itemindent=*,leftmargin=0pt]\setlength{\itemsep}{0cm}%
\item%\label{LWR-ARZ1}
If $u_\ell, u_r \in \Omega_{\rm f}$, then $\mathcal{R}_1[u_\ell,u_r] \doteq\mathcal{R}_{\rm LWR}[u_\ell,u_r]$.
\item
If $u_\ell, u_r \in \Omega_{\rm c}$, then $\mathcal{R}_1[u_\ell,u_r] \doteq\mathcal{R}_{\rm ARZ}[u_\ell,u_r]$.
\item%\label{LWR-ARZ2} 
If $(u_\ell, u_r) \in \Omega_{\rm f} \times \Omega_{\rm c}$, then we let $u_m \doteq (p^{-1}(w_\ell - v_r), v_r)$ and
\[
\mathcal{R}_1[u_\ell,u_r](x) \doteq 
\begin{cases}
u_\ell&\text{if } x<\sigma(u_\ell,u_m),\\
\mathcal{R}_{\rm ARZ}[u_m,u_r](x)&\text{if } x>\sigma(u_\ell,u_m).
\end{cases}
\]
\item%\label{LWR-ARZ3} 
If $(u_\ell, u_r) \in \Omega_{\rm c} \times \Omega_{\rm f}$, then% we let $u_m \doteq (p^{-1}(w_\ell - V_{\rm c}), V_{\rm c})$ and define
\begin{align*}
&\mathcal{R}_1[u_\ell,u_r](x) \doteq 
\begin{cases}
\mathcal{R}_{\rm ARZ}[u_\ell,u_{\rm c}(w_\ell)](x)&\text{if } x<\sigma(u_{\rm c}(w_\ell),u_{\rm f}(w_\ell)),\\
\mathcal{R}_{\rm LWR}[u_{\rm f}(w_\ell),u_r](x)&\text{if } x>\sigma(u_{\rm c}(w_\ell),u_{\rm f}(w_\ell)).
\end{cases}
\end{align*}
\end{enumerate}
The Riemann solver $\mathcal{R}_2 \colon \Omega^2 \to \Lloc1(\R;\Omega)$ is defined as follows:
\begin{enumerate}[label={($R_2$)},itemindent=*,leftmargin=0pt]\setlength{\itemsep}{0cm}%
\item%\label{LWR-ARZ2'} 
If $(u_\ell,u_r) \in \Omega_{\rm f} \times \Omega_{\rm c}$, $\rho_\ell \ne 0$ and $w_\ell < w_r$, then
\[
\mathcal{R}_2[u_\ell,u_r](x) \doteq 
\begin{cases}
u_\ell&\text{if } x<\sigma(u_\ell,u_r),\\
u_r&\text{if } x>\sigma(u_\ell,u_r),
\end{cases}
\]
otherwise $\mathcal{R}_2[u_\ell,u_r] \doteq \mathcal{R}_1[u_\ell,u_r]$.
\end{enumerate}
\end{defi}
We remark that, differently from the phase transitions from $\Omega_{\rm f}''$ to $\Omega_{\rm c}$ introduced by $\mathcal{R}_1$, those introduced by $\mathcal{R}_2$ may not satisfy the second Rankine-Hugonot condition, namely they may not conserve the generalized momentum.

In \cite[Proposition~4.2]{BenyahiaRosini01} we proved that $\mathcal{R}_1$ is $\Lloc1$-continuous and consistent in $\Omega$.
Let us recall that a Riemann solver $\mathcal{S} \colon \Omega^2 \to \L\infty(\R;\Omega)$ is consistent in the invariant domain $D \subseteq \Omega$ if for any $u_\ell, u_m, u_r \in D$ and $\bar x \in \R$:
\begin{align}\label{P2}\tag{I}
  \mathcal{S}[u_\ell,u_r](\bar x) = u_m&
  & \Rightarrow &
  &&\begin{cases}
      \mathcal{S}[u_\ell,u_m](x)= 
      \begin{cases}
          \mathcal{S}[u_\ell,u_r](x)& \hbox{if } x < \bar x ,
          \\
          u_m & \hbox{if } x \ge \bar x ,
      \end{cases}
      \\[10pt]
      \mathcal{S}[u_m,u_r](x) =
      \begin{cases}
          u_m & \hbox{if } x < \bar x ,
          \\
          \mathcal{S}[u_\ell,u_r](x)& \hbox{if } x \geq \bar x.
      \end{cases}
      \end{cases}
      \\
      \begin{rcases}
  \mathcal{S}[u_\ell,u_m](\bar x)=u_m 
  \\\label{P3}\tag{II}
  \mathcal{S}[u_m,u_r](\bar x)=u_m
  \end{rcases}&
  & \Rightarrow &
  &&\mathcal{S}[u_\ell,u_r](x)=
      \begin{cases}
      \mathcal{S}[u_\ell,u_m](x)& \hbox{if } x < \bar x ,
      \\
      \mathcal{S}[u_m,u_r](x) & \hbox{if } x \geq \bar x .
      \end{cases}
\end{align}
It is easy to prove that $\mathcal{R}_2$ is $\Lloc1$-continuous but is not consistent in $\Omega$.
Indeed, for instance, $\mathcal{R}_2$ does not satisfy \eqref{P3} with $u_\ell \in \Omega_{\rm f}''$ and $u_m,u_r \in \Omega_{\rm c}$ such that $w_\ell = w_m < w_r$ and $v_m = v_r$.

\subsection{The constrained Riemann problem}%\label{sec:tcrp}

In this section we consider the Riemann problem \eqref{eq:model} coupled with a pointwise constraint on the flux
\begin{equation}\label{eq:constraint}
Q(u(t,0^\pm)) \leq Q_0,
\end{equation}
where $Q_0 \in (0,q_{\max})$ is a fixed constant.
In general, $[(t,x) \mapsto \mathcal{R}_i[u_\ell,u_r](x/t)]$ does not satisfy \eqref{eq:constraint}.
For this reason we introduce the sets
\begin{align*}
&\mathcal{C}_i \doteq \left\{ (u_\ell,u_r) \in \Omega^2 \colon Q(\mathcal{R}_i[u_\ell,u_r](0^\pm)) \le Q_0 \right\},
&\mathcal{N}_i \doteq \left\{ (u_\ell,u_r) \in \Omega^2 \colon Q(\mathcal{R}_i[u_\ell,u_r](0^\pm)) > Q_0 \right\},
\end{align*}
and for any $(u_\ell,u_r) \in \mathcal{N}_i$, we replace $[(t,x) \mapsto \mathcal{R}_i[u_\ell,u_r](x/t)]$ by another self-similar weak solution $[(t,x) \mapsto \mathcal{R}_i^{\rm c}[u_\ell,u_r](x/t)]$ to \eqref{eq:model}, satisfying \eqref{eq:constraint} and obtained by juxtaposing weak solutions constructed by means of $\mathcal{R}_i$.
It is easy to see that $\mathcal{C}_i = \mathcal{C}^1 \cup \mathcal{C}^2 \cup \mathcal{C}^3 \cup \mathcal{C}_i^4$ and $\mathcal{N}_i = \mathcal{N}^1 \cup \mathcal{N}^2 \cup \mathcal{N}^3 \cup \mathcal{N}_i^4$, where
%\begin{gather*}
%&\mathcal{C}^1 \doteq \left\{\vphantom{p^{-1}} (u_\ell,u_r) \in \Omega_{\rm f}^2 \colon q_\ell \le Q_0 \right\},&
%&\mathcal{N}^1 \doteq \Omega_{\rm f}^2 \setminus \mathcal{C}^1,\\
%&\mathcal{C}^2 \doteq \left\{ (u_\ell,u_r) \in \Omega_{\rm c}^2 \colon p^{-1}(w_\ell-v_r) \, v_r \le Q_0 \right\},&
%&\mathcal{N}^2 \doteq \Omega_{\rm c}^2 \setminus \mathcal{C}^2,\\
%&\mathcal{C}^3 \doteq \left\{\vphantom{p^{-1}} (u_\ell,u_r) \in \Omega_{\rm c} \times \Omega_{\rm f} \colon Q(u_{\rm f}\left(w_\ell)\right) \le Q_0 \right\},&
%&\mathcal{N}^3 \doteq \left(\Omega_{\rm c} \times \Omega_{\rm f}\right) \setminus \mathcal{C}^3,\\
%&\mathcal{C}_1^4 \doteq \left\{\vphantom{p^{-1}} (u_\ell,u_r) \in \Omega_{\rm f} \times \Omega_{\rm c} \colon \min\left\{ q_\ell, p^{-1}(w_\ell-v_r) \, v_r\right\} \le Q_0 \right\},&
%&\mathcal{N}_1^4 \doteq \left(\Omega_{\rm f} \times \Omega_{\rm c}\right) \setminus \mathcal{C}_1^4,\\
%&\mathcal{C}_2^4 \doteq \left\{
%(u_\ell,u_r) \in \Omega_{\rm f} \times \Omega_{\rm c} \colon 
%\begin{array}{@{}l@{}}
%w_\ell \le w_r \text{ and } \min\left\{ q_\ell, q_r\right\} \le Q_0, 
%\text{ or} 
%\\
%w_\ell > w_r \text{ and } p^{-1}(w_\ell-v_r) \, v_r \le Q_0
%\end{array}\right\},&
%&\mathcal{N}_2^4 \doteq \left(\Omega_{\rm f} \times \Omega_{\rm c}\right) \setminus \mathcal{C}_2^4.
%\end{gather*}
\[
\begin{array}{l@{\qquad}l}
\mathcal{C}^1 \!\doteq\! \left\{\vphantom{p^{-1}} (u_\ell,u_r) \in \Omega_{\rm f}^2 \colon q_\ell \le Q_0 \right\},&
\mathcal{N}^1 \!\doteq\! \Omega_{\rm f}^2 \setminus \mathcal{C}^1,\\
\mathcal{C}^2 \!\doteq\! \left\{ (u_\ell,u_r) \in \Omega_{\rm c}^2 \colon p^{-1}(w_\ell-v_r) \, v_r \le Q_0 \right\},&
\mathcal{N}^2 \!\doteq\! \Omega_{\rm c}^2 \setminus \mathcal{C}^2,\\
\mathcal{C}^3 \!\doteq\! \left\{\vphantom{p^{-1}} (u_\ell,u_r) \in \Omega_{\rm c} \times \Omega_{\rm f} \colon Q(u_{\rm f}\left(w_\ell)\right) \le Q_0 \right\},&
\mathcal{N}^3 \!\doteq\! \left(\Omega_{\rm c} \times \Omega_{\rm f}\right) \setminus \mathcal{C}^3,\\
\mathcal{C}_1^4 \!\doteq\! \left\{\vphantom{p^{-1}} (u_\ell,u_r) \in \Omega_{\rm f} \times \Omega_{\rm c} \colon \min\left\{ q_\ell, p^{-1}(w_\ell-v_r) \, v_r\right\} \le Q_0 \right\},&
\mathcal{N}_1^4 \!\doteq\! \left(\Omega_{\rm f} \times \Omega_{\rm c}\right) \setminus \mathcal{C}_1^4,\\
\mathcal{C}_2^4 \!\doteq\! \left\{
(u_\ell,u_r) \in \Omega_{\rm f} \times \Omega_{\rm c} \colon 
\begin{array}{@{}l@{}}
w_\ell \le w_r \text{ and } \min\left\{ q_\ell, q_r\right\} \le Q_0, 
\text{ or} 
\\
w_\ell > w_r \text{ and } p^{-1}(w_\ell-v_r) \, v_r \le Q_0
\end{array}\right\},&
\mathcal{N}_2^4 \!\doteq\! \left(\Omega_{\rm f} \times \Omega_{\rm c}\right) \setminus \mathcal{C}_2^4.
\end{array}
\]
To simplify our notation, we let $\oset{*}{q} \doteq Q(\oset{*}{u})$ and $\oset{*}{w} \doteq W(\oset{*}{u})$.
\begin{defi}\label{def:R1c}
The Riemann solver $\ronec \colon \Omega^2 \to \L\infty(\R;\Omega)$ is defined as follows:
\begin{enumerate}[itemindent=*,leftmargin=0pt,label={($R_1^{\rm c}${\alph*})}]\setlength{\itemsep}{0cm}%

\item If $(u_\ell,u_r) \in \mathcal{C}_1$, then we let $\ronec[u_\ell,u_r] \doteq \mathcal{R}_1[u_\ell,u_r]$.

\item If $(u_\ell,u_r) \in \mathcal{N}_1$, then we let
\begin{equation}\label{eq:Rc}
\ronec[u_\ell,u_r](x) \doteq 
\begin{cases}
\mathcal{R}_1[u_\ell,\hat{u}](x)&\text{if }x<0,\\
\mathcal{R}_1[\check{u},u_r](x)&\text{if }x>0,
\end{cases}
\end{equation}
where $(\hat{u}, \check{u}) \in \Omega_{\rm c} \times \Omega$ satisfies
\begin{align}\label{eq:Rc1cond1}
&\mathcal{R}_1[u_\ell,\hat{u}](0^-) = \hat{u},\quad
\mathcal{R}_1[\check{u},u_r](0^+) = \check{u},\quad
\hat{q} = \check{q} \le Q_0,\quad
\hat{w} = w_\ell,
\\\label{eq:Rc1cond2}
&\text{if }(\hat{u}',\check{u}') \in \Omega^2 \setminus \{(\hat{u},\check{u})\} \text{ satisfies \eqref{eq:Rc1cond1}, then }
\hat{q}' < \hat{q}.
\end{align}
\end{enumerate}
\end{defi}
Observe that, by the consistency of $\mathcal{R}_1$, it is not restrictive to assume the first two conditions in \eqref{eq:Rc1cond1}, while the third one is required to ensure that $[(t,x) \mapsto \mathcal{R}_1^{\rm c}[u_\ell,u_r](x/t)]$ satisfies \eqref{eq:constraint} as well as the Rankine-Hugoniot conditions along $x=0$.
Let us also underline that if $(u_\ell,u_r) \in \mathcal{N}_1$, then $\hat{u}$ and $\check{u}$ must be distinct otherwise, again by the consistency of $\mathcal{R}_1$, we would have that $\mathcal{R}_1^{\rm c}[u_\ell,u_r]$ coincides with $\mathcal{R}_1[u_\ell,u_r]$, and this gives a contradiction.
Finally, in the following proposition we show that $\ronec$ is well defined.

\begin{pro}%\label{pro:Rc1}
For any $(u_\ell,u_r) \in \mathcal{N}_1$, $(\hat{u}, \check{u}) \in \Omega_{\rm c} \times \Omega$ is uniquely selected by \eqref{eq:Rc1cond1}, \eqref{eq:Rc1cond2} as follows:
\begin{enumerate}[label={($T^{\arabic*}_1$)},leftmargin=*]\setlength{\itemsep}{0cm}%
\item%\label{1} 
If $(u_\ell, u_r) \in \mathcal{N}^1 \cup \mathcal{N}^3$, then we distinguish the following cases:
\begin{enumerate}[label={($T^1_1{\alph*}$)},leftmargin=*]\setlength{\itemsep}{0cm}%
\item\label{a} 
If $Q_0 > Q(u_{\rm c}(w_\ell))$, then $\hat{u} = u_{\rm c}(w_\ell)$, $\check{q} = \hat{q}$ and $\check{u} \in \Omega_{\rm f}$.
\item%\label{d} 
If $Q_0 \le Q(u_{\rm c}(w_\ell))$, then $\hat{w} = w_\ell$, $\hat{q} = \check{q} = Q_0$ and $\check{u} \in \Omega_{\rm f}$.
\end{enumerate}
\item\label{g} 
If $(u_\ell, u_r) \in \mathcal{N}^2 \cup \mathcal{N}^4_1$, then we distinguish the following cases:
\begin{enumerate}[label={($T^2_1{\alph*}$)},leftmargin=*]\setlength{\itemsep}{0cm}%
\item\label{f} 
If $Q_0 \ge p^{-1}(W_{\rm c} - v_r) \, v_r$, then $\hat{w} = w_\ell$, $\hat{q} = \check{q} = Q_0$ and $\check{v} = v_r$.
\item%\label{t3b} 
If $Q_0 < p^{-1}(W_{\rm c} - v_r) \, v_r$, then $\hat{w} = w_\ell$, $\hat{q} = \check{q} = Q_0$ and $\check{u} \in \Omega_{\rm f}'$.
\end{enumerate}
\end{enumerate}
\end{pro}

Differently from the constrained Riemann solvers introduced in \cite{AndreianovDonadelloRosini1, BCM2order, ColomboGoatinConstraint, GaravelloGoatin1}, in the case \ref{a} we have that $(u_\ell,u_r) \in \mathcal{N}_1$ but $Q(\ronec[u_\ell,u_r](0^\pm)) \neq Q_0$.
For this reason we introduced the Riemann solver $\mathcal{R}_2$, that we use now to construct another constrained Riemann solver $\rtwoc$ such that if $(u_\ell,u_r) \in \mathcal{N}_2$ and $Q_0 \leq Q(u_{\rm c}(W_{\max}))$, then $Q(\rtwoc[u_\ell,u_r](0^\pm)) = Q_0$.

\begin{defi}\label{def:R2c}
The Riemann solver $\rtwoc \colon \Omega^2 \to \L\infty(\R;\Omega)$ is defined as follows:
\begin{enumerate}[itemindent=*,leftmargin=0pt,label={($R_2${\alph*})}]\setlength{\itemsep}{0cm}%

\item If $(u_\ell,u_r) \in \mathcal{C}_2$, then we let $\rtwoc[u_\ell,u_r] \doteq \mathcal{R}_2[u_\ell,u_r]$.

\item If $(u_\ell,u_r) \in \mathcal{N}^3$ and $Q_0 > Q(u_{\rm c}(w_\ell))$, then we let
\begin{equation*}%\label{eq:Rc2bis}
\rtwoc[u_\ell,u_r](x) \doteq 
\begin{cases}
\mathcal{R}_2[u_\ell,u_{\rm f}(w_\ell)](x)&\text{if } x < \sigma(u_{\rm f}(w_\ell),\hat{u}),\\
\hat{u}&\text{if }\sigma(u_{\rm f}(w_\ell),\hat{u}) < x < 0,\\
\mathcal{R}_2[\check{u},u_r](x)&\text{if }x>0,
\end{cases}
\end{equation*}
otherwise we let
\begin{equation*}%\label{eq:Rc2}
\rtwoc[u_\ell,u_r](x) \doteq 
\begin{cases}
\mathcal{R}_2[u_\ell,\hat{u}](x)&\text{if }x<0,\\
\mathcal{R}_2[\check{u},u_r](x)&\text{if }x>0,
\end{cases}
\end{equation*}
where, in both cases, $(\hat{u}, \check{u}) \in \Omega_{\rm c} \times \Omega$ satisfies
\begin{align}\label{eq:Rc2cond1}
&%\mathcal{R}_2[u_\ell,\hat{u}](0^-) = \hat{u},\quad
\mathcal{R}_2[\check{u},u_r](0^+) = \check{u},\quad
\hat{q} = \check{q} = \min\{Q_0,Q(u_{\rm c}(W_{\rm max}))\},\quad
\hat{w} \ge w_\ell,
\\\label{eq:Rc2cond2}
&\text{if }(\hat{u}',\check{u}') \in \Omega_{\rm c} \times \Omega \setminus \{(\hat{u},\check{u})\} \text{ satisfies \eqref{eq:Rc2cond1}, then }
\hat{w}' > \hat{w}
\text{ or }
\check{w}' < \check{w}.
\end{align}
\end{enumerate}
\end{defi}

In the following proposition we show that $\rtwoc$ is well defined.

\begin{pro}%\label{prop:Rc2}
For any $(u_\ell,u_r) \in \mathcal{N}_2$, $(\hat{u}, \check{u}) \in \Omega_{\rm c} \times \Omega$ is uniquely selected by \eqref{eq:Rc2cond1}, \eqref{eq:Rc2cond2} as follows:
\begin{enumerate}[label={($T^{\arabic*}_2$)},leftmargin=*]\setlength{\itemsep}{0cm}%
\item If $(u_\ell, u_r) \in \mathcal{N}^1 \cup \mathcal{N}^3$, then we distinguish the following cases:
\begin{enumerate}[label={($T^1_2${\alph*})},leftmargin=*]\setlength{\itemsep}{0cm}%
\item If $Q_0 > Q(u_{\rm c}(W_{\max}))$, then $\hat{u} = u_{\rm c}(W_{\max})$, $\check{q} = \hat{q}$ and $\check{u} \in \Omega_{\rm f}$.
\item If $Q_0 \leq Q(u_{\rm c}(W_{\max}))$, then $\hat{w} = \max\left \{w_\ell, V_{\rm c}+p\left (Q_0/V_{\rm c}\right )\right \}$, $\hat{q} = \check{q} = Q_0$ and $\check{u} \in \Omega_{\rm f}$.
\end{enumerate}
\item If $(u_\ell, u_r) \in \mathcal{N}^2 \cup \mathcal{N}^4_2$, then we distinguish the following cases:
\begin{enumerate}[label={($T^2_2${\alph*})},leftmargin=*]\setlength{\itemsep}{0cm}%
\item If $Q_0 \ge p^{-1}(W_{\rm c} - v_r) \, v_r$, then $\hat{w} = w_\ell$, $\hat{q} = \check{q} = Q_0$ and $\check{v} = v_r$.
\item If $Q_0 < p^{-1}(W_{\rm c} - v_r) \, v_r$, then $\hat{w} = w_\ell$, $\hat{q} = \check{q} = Q_0$ and $\check{u} \in \Omega_{\rm f}'$.
\end{enumerate}
\end{enumerate}
\end{pro}

Moreover, by definition we immediately have the following

\begin{pro}\label{pro:Rush1}
For any $(u_\ell,u_r) \in \Omega^2$ we have that both $[(t,x) \mapsto \ronec[u_\ell,u_r](x/t)]$ and $[(t,x) \mapsto \rtwoc[u_\ell,u_r](x/t)]$ are weak solutions of \eqref{eq:model} and satisfy \eqref{eq:constraint} for a.e.~$t\in\R_+$.
Moreover, among the self similar weak solutions $u'$ to \eqref{eq:model} of the form \eqref{eq:Rc} and satisfying \eqref{eq:Rc1cond1}, $u \doteq \mathcal{R}_1^{\rm c}[u_\ell,u_r]$ is the only one that maximizes the flow through ${x=0}$, namely $Q(u'(t,0^\pm)) \le Q(u(t,0^\pm))$, with the equality holding if and only if $u' = u$.
\end{pro}

It is easy to prove that in general both $\mathcal{R}_1^{\rm c}$ and $\mathcal{R}_2^{\rm c}$ fail to be consistent.

\begin{pro}\label{pro:cons}
In general, both $\ronec$ and $\rtwoc$ satisfy neither \eqref{P2} nor \eqref{P3} in $\Omega$.
\end{pro}
\begin{proof}
For any $u_\ell, u_m, u_r \in \Omega$ such that $u_m = u_r$ and $q_r >Q_0$, by the finite speed of propagation of the waves, there exists $\bar{x} >0$ such that $\ronec[u_\ell,u_r](\bar{x}) = \rtwoc[u_\ell,u_r](\bar{x}) = u_r$. Then the property $\ronec[u_r,u_r](x) = \rtwoc[u_r,u_r](x) = u_r$ for any $x < \bar{x}$ required in \eqref{P2} cannot be satisfied because otherwise $Q(\ronec[u_r,u_r](0^\pm)) = Q(\rtwoc[u_r,u_r](0^\pm)) = q_r > Q_0$ and this gives a contradiction because $\ronec[u_r,u_r]$ and $\rtwoc[u_r,u_r]$ satisfy \eqref{eq:constraint}, see Proposition~\ref{pro:Rush1}.
Moreover, if $Q_0\in [Q(u_{\rm c}(W_{\rm c})), Q(u_{\rm c}(W_{\max}))]$, then we can take $u_\ell, u_r \in \Omega_{\rm f}$, $u_m \in \Omega_{\rm c}$ with $w_\ell = w_m$ and $Q(u_m) = Q_0 = Q(u_r)$, and see that \eqref{P3} is not satisfied by both $\mathcal{R}_1^{\rm c}$ and $\mathcal{R}_2^{\rm c}$.
\end{proof}

We conclude this section by considering the total variation of the two constrained Riemann solvers in the Riemann invariant coordinates.
We provide two examples showing that in general the comparison of their total variation can go in both ways. 
This suggests that the total variation is not a relevant selection criteria for choosing a wave-front tracking algorithm based on one or the other Riemann solver.

\begin{example}
If there exist $\bar{u}, u^* \in \Omega_{\rm f}$ and $Q_0$ such that $Q(u_{\rm f}(W_{\rm c})) < Q(u^*) = Q(u_{\rm c}(W(\bar{u}))) < Q_0 < Q(\bar{u}) < Q(u_{\rm c}(W_{\max}))$ and $W(\bar{u}) - W(u^*) > \check{w}_2 - W(u_0)$, where $u_0 \in \Omega_{\rm f}$ is implicitly defined by $Q(u_0) = Q_0$ and $\check{w}_2 \doteq V_{\rm c} + p\left(Q_0/V_{\rm c}\right)$, then $\tv(V \circ \ronec[\bar{u},u_r])= 2[V(u^*) - V_{\rm c}] > \tv(V \circ \rtwoc[\bar{u},u_r])= 2[V(u_0) - V_{\rm c}]$
and $\tv(W \circ \ronec[\bar{u},u_r])= 2[W(\bar{u}) - W(u^*)] > \tv(W \circ \rtwoc[\bar{u},u_r])= 2[\check{w}_2 - W(u_0)]$.
\end{example}

\begin{example}
If there exist $(u_\ell , u_r) \in \Omega_{\rm c} \times \Omega_{\rm f}$ and $Q_0$ such that $v_\ell = V_{\rm c}$ and $Q(u_{\rm f}(W_{\rm c})) < Q(u_\ell) = Q(u_r) < Q_0 < Q(u_{\rm f}(w_\ell)) < Q(u_{\rm c}(W_{\max}))$, then $\tv(V \circ \ronec[u_\ell,u_r]) = v_r - V_{\rm c}<\tv(V \circ \rtwoc[u_\ell,u_r]) = v_r + 2 V(u_{\rm f}(w_\ell)) - 3 V_{\rm c}$ and $\tv(W \circ \ronec[u_\ell,u_r]) = w_\ell - w_r < \tv(W \circ \rtwoc[u_\ell,u_r]) = 2 \check{w}_2  - w_\ell - w_r$, where $\check{w}_2 \doteq V_{\rm c} + p\left(Q_0/V_{\rm c}\right)$.
\end{example}

\section{Basic properties of the constrained Riemann solvers}\label{sec:3}

\subsection{Basic properties of \texorpdfstring{$\ronec$}{}}%\label{sec:properties}

\begin{pro}\label{pro:consistencyr1c}
$\ronec$ is consistent in the invariant domain $\mathcal{D}_1 \doteq \{u \in \Omega \colon  Q(u) \leq Q_0 \}$.
Moreover, if $\mathcal{D}$ is an invariant domain and is not contained in $\mathcal{D}_1$, then $\ronec$ is not consistent in $\mathcal{D}$.
\end{pro}\noindent
The proof is rather technical and is therefore deferred to Section~\ref{sec:consistencyr1c}.

\begin{pro}\label{pro:continuity}
$\mathcal{R}_1^{\rm c}$ is $\Lloc1$-continuous in $\Omega^2$ if and only if $Q_0 \le Q(u_{\rm c}(W_{\rm c}))$.
If $Q_0 > Q(u_{\rm c}(W_{\rm c}))$, then $\mathcal{R}_1^{\rm c}$ is $\Lloc1$-continuous in $\Omega^2 \setminus(\mathcal{C} \cap \overline{\mathcal{N}^1})$ and is not $\Lloc1$-continuous in any point of $\mathcal{C} \cap \overline{\mathcal{N}^1}$.
\end{pro}

\begin{proof}
Assume that $Q_0 > Q(u_{\rm c}(W_{\rm c}))$ and let $u_0 \in \Omega_{\rm f}$ be such that $Q(u_0) = Q_0$.
Then it suffices to take $u_\ell = u_0$, $u_r = (R_{\rm f}'',V_{\min})$ and $u_\ell^n \in \Omega_{\rm f}$ with $\rho_\ell^n \doteq \rho_0+1/n$.
Indeed in this case $(u_\ell^n)_n$ converges to $u_\ell$ but $\mathcal{R}_1^{\rm c}[u_\ell^n,u_r]$ does not converge to $\mathcal{R}_1^{\rm c}[u_\ell,u_r]$ in $\Lloc1(\R;\Omega)$.
More precisely, $\mathcal{R}_1^{\rm c}[u_\ell,u_r] \equiv u_\ell$ in $\R_-$ and by \ref{a} the restriction of $\mathcal{R}_1^{\rm c}[u_\ell^n,u_r]$ to $\R_-$ converges to
\[
\begin{cases}
u_\ell&\text{if }x<\sigma(u_\ell,u_{\rm c}(w_\ell)),\\
u_{\rm c}(W_{\rm c})&\text{if }\sigma(u_\ell,u_{\rm c}(w_\ell)) < x < 0.
\end{cases}
\]

Proving that $\mathcal{R}_1^{\rm c}$ is $\Lloc1$-continuous in $\Omega^2$ in any other situation is now a matter of showing that $\mathcal{R}_1[u^n_\ell,\hat{u}^n] \to \mathcal{R}_1^{\rm c}[u_\ell,u_r]$ pointwise in $\{x<0\}$, $\mathcal{R}_1[\check{u}^n,u^n_r] \to \mathcal{R}_1^{\rm c}[u_\ell,u_r]$ pointwise in $\{x>0\}$, and applying the dominated convergence theorem of Lebesgue.
For this, it suffices to observe that either $\hat{u}^n \to \mathcal{R}_1[u_\ell,u_r](0^-)$ and the result follows then by the continuity of $\mathcal{R}_1$, or $\sigma(u^{n}_\ell, \hat{u}^{n}) \to 0$ and $\mathcal{R}_1[u_\ell,u_r]$ is constant equal to $u_\ell$ in $\{x <0\}$ and we obtain therefore again that $\mathcal{R}_1[u^n_\ell,\hat{u}^n] \to \mathcal{R}_1[u_\ell,u_r]$ pointwise in $\{x<0\}$. A similar analysis proves that $\mathcal{R}_1[\check{u}^n,u^n_r] \to \mathcal{R}_1[u_\ell,u_r]$ pointwise in $\{x>0\}$.
\end{proof}

In the next proposition we study the invariant domains of $\ronec$.
However, since in \cite{BenyahiaRosini01} we did not consider the invariant domains of $\mathcal{R}_1$, let us first point out that for any $0\leq \rho_{\min} < \rho_{\max} \leq R_{\rm f}''$, $0\leq v_{\min} < v_{\max} \leq V_{\rm c}$ and $W_{\rm c}\leq w_{\min} < w_{\max} \leq W_{\max}$ the following sets are invariant domains for $\mathcal{R}_1$.
\begin{gather*}%\tag{{$I_1a$}}
%\label{inv-R-a}
\left\{\vphantom{\Omega_{\rm f}''} u \in \Omega_{\rm f} \colon \rho_{\min} \leq \rho \leq \rho_{\max} \right\}
\\
%\tag{{$I_1b$}}
%\label{inv-R-b}
\left\{\vphantom{\Omega_{\rm f}''} u \in \Omega_{\rm c} \colon w_{\min} \leq W(u) \leq w_{\max} ,~ v_{\min} \leq v \leq v_{\max} \right\}
\\
\left\{ u \in \Omega_{\rm f}'' \colon \rho_{\rm f}(w_{\min}) \leq \rho \leq \rho_{\rm f}(w_{\max}) \right\} \cup \left\{\vphantom{\Omega_{\rm f}''} u \in \Omega_{\rm c} \colon w_{\min} \leq W(u) \leq w_{\max} ,~ v \geq v_{\min} \right\}
\\
\rho_{\min}<R_{\rm f}' ~\Rightarrow~
\left\{\vphantom{\Omega_{\rm f}''} u \in \Omega_{\rm f} \colon \rho_{\min} \leq \rho \leq \rho_{\rm f}(w_{\max}) \right\} \cup \left\{\vphantom{\Omega_{\rm f}''} u \in \Omega_{\rm c} \colon W(u) \leq w_{\max} ,~ v \geq v_{\min} \right\}
\end{gather*}
%\begin{enumerate}[label={($I_1\alph*$)},itemindent=*,leftmargin=0pt]
%\item\label{inv-R-a}
%If $0\leq \rho_{\min} < \rho_{\max} \leq R_{\rm f}''$, then $\{ u \in \Omega_{\rm f} \colon \rho_{\min} \leq \rho \leq \rho_{\max} \}$ is an invariant domain for $\mathcal{R}_1$.
%\item\label{inv-R-b}
%If $W_{\rm c}\leq w_{\min} < w_{\max} \leq W_{\max}$ and $0\leq v_{\min} < v_{\max} \leq V_{\rm c}$, then $\{ u \in \Omega_{\rm c} \colon w_{\min} \leq W(u) \leq w_{\max} ,~ v_{\min} \leq v \leq v_{\max} \}$ is an invariant domain for $\mathcal{R}_1$.
%\item%\label{inv-R-c}
%If $W_{\rm c}\leq w_{\min} < w_{\max} \leq W_{\max}$ and $0\leq v_{\min} < V_{\rm c}$, then $\{ u \in \Omega_{\rm f}'' \colon \rho_{\rm f}(w_{\min}) \leq \rho \leq \rho_{\rm f}(w_{\max}) \} \cup \{ u \in \Omega_{\rm c} \colon w_{\min} \leq W(u) \leq w_{\max} ,~ v \geq v_{\min} \}$ is an invariant domain for $\mathcal{R}_1$.
%\item%\label{inv-R-d}
%If $0 \leq \rho_{\min} < \rho_{\rm f}(W_{\rm c})$, $W_{\rm c} < w_{\max} \leq W_{\max}$ and $0\leq v_{\min} < V_{\rm c}$, then $\{ u \in \Omega_{\rm f} \colon \rho_{\min} \leq \rho \leq \rho_{\rm f}(w_{\max}) \} \cup \{ u \in \Omega_{\rm c} \colon W(u) \leq w_{\max} ,~ v \geq v_{\min} \}$ is an invariant domain for $\mathcal{R}_1$.
%\end{enumerate}

\begin{pro}[Invariant domains for $\ronec$]\label{pro:inv-Rc}~
\begin{enumerate}[label={($I_1^{\rm c}\alph*$)},itemindent=*,leftmargin=0pt]\setlength{\itemsep}{0cm}%

\item\label{inv-Rc-a} 
If $Q_0 < Q(u_{\rm c}(W_{\max}))$, then $\Omega_{\rm f} \cup \{u \in \Omega_{\rm c} \colon Q(u) \leq Q_0 \le  p^{-1}(W_{\max}-v) \, v \}$ is the smallest invariant domain containing $\Omega_{\rm f}$.

\item\label{inv-Rc-b} 
If $Q_0 \ge Q(u_{\rm c}(W_{\max}))$, then $\Omega_{\rm f} \cup \{ u \in \Omega_{\rm c} \colon v = V_{\rm c} \}$ is the smallest invariant domain containing $\Omega_{\rm f}$.

\item\label{inv-Rc-c} 
If $Q_0 \geq Q(u_{\rm c}(W_{\rm c}))$, then $\Omega_{\rm c}$ is the smallest invariant domain containing $\Omega_{\rm c}$.

\item\label{inv-Rc-d} 
If $Q_0 < Q(u_{\rm c}(W_{\rm c}))$, then $\Omega_{\rm c} \cup \{ u \in \Omega'_{\rm f} \colon Q(u) = Q_0 \}$ is the smallest invariant domain containing $\Omega_{\rm c}$.
\end{enumerate}
\end{pro}\noindent
The proof is postponed to Section~\ref{sec:inv-Rc}.

\subsection{Basic properties of \texorpdfstring{$\rtwoc$}{}}%\label{sec:properties}

Concerning $\rtwoc$, in general no significant positive result for consistency can be expected because $\mathcal{R}_2$ is not consistent.
In the following proposition we prove that the consistency of $\rtwoc$ is guaranteed only in very special invariant domains.

\begin{pro}%\label{pro:consRc2}
$\rtwoc$ is consistent in the invariant domain $\mathcal{D}_2 \doteq \{u \in \Omega_{\rm f} \colon Q(u) \le Q_0 \}$.
Moreover, if $Q(u_{\rm c}(W_{\rm c})) \le Q_0 \le Q(u_{\rm c}(W_{\max}))$ and $w_0 \in [W_{\rm c},W_{\max}]$ is such that $Q(u_{\rm c}(w_0)) = Q_0$, then for any fixed $\bar{w} \in [W_{\rm c},w_0]$ we have that $\rtwoc$ is consistent also in the invariant domains $\mathcal{D}_2' \doteq \{u \in \Omega_{\rm c} \colon W(u) \ge w_0 ,~ Q(u) \le Q_0 \}$ and $\mathcal{D}_2'' \doteq \{u \in \Omega \colon W(u) = \bar{w} ,~ Q(u) \le Q_0 \}$.
Moreover, $\rtwoc$ is not consistent in any other invariant domain containing either $\mathcal{D}_2$, or  $\mathcal{D}_2'$, or else $\mathcal{D}_2''$.
\end{pro}
\begin{proof}
To prove the first part, it suffices to observe that $\rtwoc = \mathcal{R}_1$ in $\mathcal{D}_2$,  $\mathcal{D}_2'$ and $\mathcal{D}_2''$.
Then the maximality property is a direct consequence of the fact that $\rtwoc$ is not consistent in an invariant domain $\mathcal{D}$ in any of the following cases:
\begin{enumerate}[label={($\alph*$)},itemindent=*,leftmargin=0pt]\setlength{\itemsep}{0cm}%
\item\label{rgb1}
$\exists \, u \in \mathcal{D}$ such that $Q(u) > Q_0$.
\item\label{rgb2}
$\exists \, u_1 \in \mathcal{D} \cap \Omega_{\rm f}''$, $\exists \, u_2 \in \mathcal{D} \cap \Omega_{\rm c}$ such that $w_2(u_1) > w_2(u_2)$.
\item\label{rgb3}
$\exists \, u_1 \in \mathcal{D} \cap \Omega_{\rm f}''$, $\exists \, u_2, u_3 \in \mathcal{D} \cap \Omega_{\rm c}$ such that $w_2(u_1) = w_2(u_2) < w_2(u_3)$.
\end{enumerate}
Indeed, in the case \ref{rgb1}, following exactly the same arguments used in Proposition~\ref{pro:cons}, we can show that $\rtwoc$ does not satisfy \eqref{P2} in $\mathcal{D}$.
As a consequence, in the following we assume that $D \subseteq \{u \in \Omega \colon  Q(u) \leq Q_0 \}$.
In the case \ref{rgb2}, $u_3 \in \Omega_{\rm c}$ and $u_4 \in \Omega_{\rm f}''$ defined by $w_2(u_3) = w_2(u_1)$ and $w_2(u_4) = w_2(u_2)$ both belong to $D$.
Hence we can consider \eqref{P3} with $u_\ell = u_4$, $u_m = u_3$ and $u_r = u_2$.
Finally, in the case \ref{rgb3}, it is not restrictive to assume that $v_2 = v_3$.
Then we can consider \eqref{P3} with $u_\ell = u_1$, $u_m = u_2$ and $u_r = u_3$.\qedhere
\end{proof}

\begin{pro}
$\rtwoc$ is $\Lloc1$-continuous in $\Omega^2$.
\end{pro}
\begin{proof}
\begin{itemize}[itemindent=*,leftmargin=0pt]\setlength{\itemsep}{0cm}%
\item If $u_\ell,u_r \in \Omega_{\rm f}$, then the $\Lloc1$-continuity of $\rtwoc$ follows from the continuity of $\sigma(u_\ell,\hat{u})$, $\sigma(\check{u},u_r)$ with respect to $(u_\ell,u_r)$ and from the continuity of $\mathcal{R}_{\rm LWR}$.
\item If $u_\ell,u_r \in \Omega_{\rm c}$, then $\rtwoc[u_\ell,u_r] = \ronec[u_\ell,u_r]$ and the continuity follows from Proposition \ref{pro:continuity}.
\item If $(u_\ell, u_r) \in \Omega_{\rm c} \times \Omega_{\rm f}$ and $Q(u_{\rm c}(w_\ell)) > Q_0$, then $\rtwoc[u_\ell,u_r] = \ronec[u_\ell,u_r]$ and the continuity follows from Proposition \ref{pro:continuity}.
\item If $(u_\ell, u_r) \in \Omega_{\rm c} \times \Omega_{\rm f}$ and $Q(u_{\rm c}(w_\ell)) < Q_0$, then the continuity follows from the continuity of $u_{\rm c}(w_\ell)$, $u_{\rm f}(w_\ell)$, $\sigma(u_{\rm f}(w_\ell), \hat{u})$ with respect to $u_\ell$ and Proposition~\ref{pro:continuity}.
\item If $(u_\ell, u_r) \in \Omega_{\rm c} \times \Omega_{\rm f}$ and $Q(u_{\rm c}(w_\ell)) = Q_0$, then it suffices to consider for $n$ sufficiently large $u_\ell^n$ defined by $v_\ell^n=v_\ell$ and $w_\ell^n = w_\ell - 1/n$.
Then $u_\ell^n \to u_\ell$ and $Q(u_{\rm c}(w_\ell^n)) < Q_0$.
Hence, roughly speaking, $\rtwoc[u_\ell^n,u_r]$ has two phase transitions, one from $u_{\rm c}(w_\ell^n)$ to $u_{\rm f}(w_\ell^n)$ and one from $u_{\rm f}(w_\ell^n)$ to $\hat{u}^n=\hat{u}$, that are not performed by $\rtwoc[u_\ell,u_r]$.
However, both $\sigma(u_{\rm c}(w_\ell^n),u_{\rm f}(w_\ell^n))$ and $\sigma(u_{\rm f}(w_\ell^n), \hat{u})$ converge to $\sigma(u_{\rm f}(w_\ell), \hat{u})$.
Therefore also in this case we have that $\rtwoc[u_\ell^n, u_r] \to \rtwoc[u_\ell, u_r]$ in $\Lloc1$.
\item Assume $(u_\ell, u_r) \in \Omega_{\rm f} \times \Omega_{\rm c}$, then the continuity in that case comes from the continuity of $\sigma(u_\ell,u_r)$, $\sigma(u_\ell,\hat{u})$, $\sigma(\check{u},u_r)$ and $\ronec$ with respect to $(u_\ell,u_r)$.\qedhere
\end{itemize}
\end{proof}

\begin{pro}[Invariant domains for $\rtwoc$]%\label{pro:inv-Rtwoc}
~
\begin{enumerate}[label={($I_2^{\rm c}\alph*$)},itemindent=*,leftmargin=0pt]\setlength{\itemsep}{0cm}%

\item%\label{inv-Rc-a} 
If $Q_0 < Q(u_{\rm c}(W_{\max}))$, then $\Omega_{\rm f} \cup \{u \in \Omega_{\rm c} \colon Q(u) \leq Q_0 ,~ p^{-1}(W_{\max}-v) \, v \geq Q_0 \}$ is an invariant domain for $\rtwoc$.

\item%\label{inv-Rc-b} 
If $Q(u_{\rm c}(W_{\max})) \leq Q_0$, then $\Omega_{\rm f} \cup \{ (V_{\rm c},W_{\max}) \}$ is an invariant domain for $\rtwoc$. 

\item%\label{inv-Rc-c} 
If $Q_0 \geq Q(u_{\rm c}(W_{\rm c}))$, then $\Omega_{\rm c}$ is an invariant domain for $\rtwoc$.

\item%\label{inv-Rc-d} 
If $Q_0 < Q(u_{\rm c}(W_{\rm c}))$, then $\Omega_{\rm c} \cup \{ u \in \Omega'_{\rm f} \colon Q(u) = Q_0 \}$ is an invariant domain for $\rtwoc$.
\end{enumerate}
\end{pro}

\begin{proof}
\begin{enumerate}[label={($I_2^{\rm c}\alph*$)},itemindent=*,leftmargin=0pt]\setlength{\itemsep}{0cm}%
\item Whenever $u_\ell, u_r \in \Omega_{\rm f} \cup \{u \in \Omega_{\rm c} \colon Q(u) \leq Q_0 ,~ p^{-1}(W_{\max}-v) \, v \geq Q_0 \}$ are such that $\ronec[u_\ell,u_r] = \rtwoc[u_\ell,u_r]$ it is clear by Proposition \ref{pro:inv-Rc} that $ \rtwoc[u_\ell,u_r] \subset \Omega_{\rm f} \cup \{u \in \Omega_{\rm c} \colon Q(u) \leq Q_0 ,~ p^{-1}(W_{\max}-v) \, v \geq Q_0 \}$.
If $\ronec[u_\ell,u_r] \neq \rtwoc[u_\ell,u_r]$  and $u_\ell \in \Omega_{\rm f}, u_r \in \Omega_{\rm c}$ then $\rtwoc[u_\ell,u_r] \subset \ronec[u_\ell,u_r]$ in which case it is obvious that $\Omega_{\rm f} \cup \{u \in \Omega_{\rm c} \colon Q(u) \leq Q_0 ,~ p^{-1}(W_{\max}-v) \, v \geq Q_0 \}$ is an invariant domain for $\rtwoc$. Otherwise $\ronec[u_\ell,u_r] \neq \rtwoc[u_\ell,u_r]$  and $u_\ell, u_r \in \Omega_{\rm f}$ in which case $\rtwoc[u_\ell,u_r] \subset \Omega_{\rm f} \cup \{u \in \Omega_{\rm c} \colon Q(u) \leq Q_0 ,~ p^{-1}(W_{\max}-v) \, v \geq Q_0 \}$ because $\{\hat{u},\check{u}\} \in \Omega_{\rm f} \cup \{u \in \Omega_{\rm c} \colon Q(u) \leq Q_0 ,~ p^{-1}(W_{\max}-v) \, v \geq Q_0 \}$.
\item In this case we distinguish the following cases:
\begin{itemize}\setlength{\itemsep}{0cm}%
\item If $u_\ell,u_r \in \Omega_{\rm f}$ and $Q(u_\ell) \leq Q_0$, or if $u_\ell = (V_{\rm c}, W_{\max})$ then $\rtwoc[u_\ell,u_r]= \ronec[u_\ell,u_r]$ and it is clear that $\rtwoc[u_\ell,u_r] \subset \Omega_{\rm f} \cup \{ (V_{\rm c},W_{\max}) \}$.
\item If $u_\ell ,u_r \in \Omega_{\rm f}$ and $Q(u_\ell) > Q_0$ then $\rtwoc[u_\ell,u_r]( x<0 )$ consists of a discontinuity between $u_\ell$ and $(V_{\rm c},W_{\max})$ while $\rtwoc[u_\ell,u_r]( x>0 )= \ronec[u_\ell,u_r]( x>0 )$ and we again have $\rtwoc[u_\ell,u_r] \subset \Omega_{\rm f} \cup \{ (V_{\rm c},W_{\max}) \}$. 
\item If $u_r= (V_{\rm c},W_{\max})$, then $\rtwoc[u_\ell,u_r]$ consists of a discontinuity between $u_\ell$ and $(V_{\rm c},W_{\max})$.
\end{itemize}
\item In this case we have for any $u_\ell,u_r \in \Omega_{\rm c}$, $\rtwoc[u_\ell,u_r]=\ronec[u_\ell,u_r]$.
\item In this case the only time when $\rtwoc[u_\ell,u_r] \neq \ronec[u_\ell,u_r]$ is when $u_\ell \in \Omega_{\rm f}$ with $Q(u_\ell)=Q_0$, in which case $\rtwoc[u_\ell,u_r]$ consists of a single discontinuity between $u_\ell$ and $u_r$ and therefore $\rtwoc[u_\ell,u_r] \subset \Omega_{\rm c} \cup \{ u \in \Omega'_{\rm f} \colon Q(u) = Q_0 \}$.\qedhere
\end{enumerate}
\end{proof}

\section{Technical section}

\subsection{Proof of Proposition~\ref{pro:consistencyr1c}}\label{sec:consistencyr1c}
To simplify the exposition of the proof, we divide it into the following two lemmas.

%\begin{lem}
%$\mathcal{D}_1$ is an invariant domain for $\ronec$.
%\end{lem}
%\begin{proof}
%By contradiction, assume that there is $u_\ell,u_r \in \mathcal{D}_1$ and $\bar{x}\in \R$ such that $u_m \doteq \ronec[u_\ell, u_r](\bar{x}) \notin \mathcal{D}_1$, namely $Q(u_m) > Q_0$.
%By Proposition~\ref{pro:Rush1} we have that $\bar{x}\ne0$.
%If $\bar{x}>0$, then we have that\ldots
%\begin{center}
%\ldots
%\end{center}
%\end{proof}
%\begin{proof}
%By contradiction, assume that there is $u_\ell,u_r \in \mathcal{D}_1$ and $\bar{x}\in \R$ such that $u_m \doteq \ronec[u_\ell, u_r](\bar{x}) \notin \mathcal{D}_1$, namely $Q(u_m) > Q_0$.
%By \eqref{eq:Rc1cond1} we have that $\bar{x}\ne0$.
%If $\bar{x}>0$, then by Lemma~\ref{lem:Rc1I} we have that $\ronec[u_m,u_r](x) = u_m$ for any $x< \bar{x}$. 
%In particular this implies that $Q(\ronec[u_m,u_r](0^\pm)) = Q(u_m) > q_0$, and this gives a contradiction.
%A similar argument holds for $\bar{x} <0$.
%\end{proof}

\begin{lem}%\label{lem:Rc1I}
$\ronec$ satisfies \eqref{P2} in $\mathcal{D}_1$.
\end{lem}
\begin{proof}
Fix $u_\ell, u_m, u_r \in \mathcal{D}_1$, $\bar{x} \in \R$ such that $\ronec[u_\ell,u_r](\bar{x}) = u_m$. 
We distinguish the following cases:
\begin{itemize}[itemindent=*,leftmargin=0pt]\setlength{\itemsep}{0cm}%
\item 
If $(u_\ell,u_r) \in \mathcal{N}^2 \cup \mathcal{N}^3$ and $\bar{x} \leq 0$ (the case $\bar{x} \geq 0$ is analogous), then $ u_m = \ronec[u_\ell,u_r](\bar{x}) = \mathcal{R}_1[u_\ell,\hat{u}](\bar{x})$ and by exploiting the consistency of $\mathcal{R}_1$ we have
\begin{align*}
\ronec[u_\ell,u_m](x) &=
\mathcal{R}_1[u_\ell,u_m](x) =
\begin{cases}
\mathcal{R}_1[u_\ell,\hat{u}](x) & \hbox{if } x < \bar x 
\\
u_m & \hbox{if } x \ge \bar x 
\end{cases} =
\begin{cases}
\ronec[u_\ell,u_r](x) & \hbox{if } x < \bar x ,
\\
u_m & \hbox{if } x \ge \bar x ,
\end{cases}
\\
\ronec[u_m,u_r](x) &=
\begin{cases}
\mathcal{R}_1[u_m,\hat{u}](x) & \hbox{if } x < 0
\\
\mathcal{R}_1[\check{u},u_r](x) & \hbox{if } x \geq 0
\end{cases}
 =
\begin{cases}
u_m & \hbox{if } x< \bar{x}
\\
\mathcal{R}_1[u_\ell,\hat{u}](x) & \hbox{if } \bar{x} \leq x < 0
\\
\mathcal{R}_1[\check{u},u_r](x) & \hbox{if } x\ge 0
\end{cases} =
\begin{cases}
u_m & \hbox{if } x< \bar{x} ,
\\
\ronec[u_\ell,u_r](x) & \hbox{if } x \ge \bar{x} .
\end{cases}
\end{align*}
\item 
If $(u_\ell,u_r) \notin \mathcal{N}^2 \cup \mathcal{N}^3$, then $(u_\ell,u_r), (u_\ell,u_m), (u_m,u_r) \in \mathcal{C}_1$ and \eqref{P2} comes from the consistency of $\mathcal{R}_1$.
\qedhere       
\end{itemize}
\end{proof}

\begin{lem}
$\ronec$ satisfies \eqref{P3} in $\mathcal{D}_1$.
\end{lem}
\begin{proof}
Fix $u_\ell, u_m, u_r \in \mathcal{D}_1$, $\bar{x} \in \R$ and assume that $\ronec[u_\ell,u_m](\bar{x}) = u_m = \ronec[u_m,u_r](\bar{x})$.
Then we consider the following cases:
\begin{itemize}[itemindent=*,leftmargin=0pt]\setlength{\itemsep}{0cm}%
\item 
If $u_\ell, u_r \in \Omega_{\rm f}$, then also $u_m \in \Omega_{\rm f}$. 
Hence \eqref{P3} follows from the consistency of $\mathcal{R}_1$, because in this case $(u_\ell,u_r), (u_\ell,u_m), (u_m,u_r) \in \mathcal{C}^1 \subseteq \mathcal{C}_1$.
\item
If $u_\ell,u_r \in \Omega_{\rm c}$, then also $u_m \in \Omega_{\rm c}$ with $w_m = w_\ell$, $v_m = v_r$.
Hence \eqref{P3} follows from the consistency of $\mathcal{R}_1$, because in this case $(u_\ell,u_r), (u_\ell,u_m), (u_m,u_r) \in \mathcal{C}^2 \subseteq \mathcal{C}_1$.
\item 
If $(u_\ell, u_r) \in \Omega_{\rm c} \times \Omega_{\rm f}$, then $u_m \in \Omega_{\rm f}$ and $q_m = Q(\hat{u})$, with $\hat{w} = w_\ell$, $\hat{v} = V_{\rm c}$. 
In this case it is easy to prove \eqref{P3}.
\item 
If $(u_\ell, u_r) \in \Omega_{\rm f} \times \Omega_{\rm c}$, then $v_m = v_r$ and $w_m = w_\ell$.
Hence \eqref{P3} follows from the consistency of $\mathcal{R}_1$, because in this case $(u_\ell,u_r), (u_\ell,u_m) \in \mathcal{C}_1^4 \subseteq \mathcal{C}_1$ and $(u_m,u_r) \in \mathcal{C}^2 \subseteq \mathcal{C}_1$.\qedhere
\end{itemize}
\end{proof}

We conclude the section by observing that the maximality of $\mathcal{D}_1$ follows from the proof of Proposition~\ref{pro:cons}.

\subsection{Proof of Proposition~\ref{pro:inv-Rc}}\label{sec:inv-Rc}

\begin{enumerate}[label={($I_1^{\rm c}\alph*$)},itemindent=*,leftmargin=0pt]\setlength{\itemsep}{0cm}%

\item[\ref{inv-Rc-a}] 
We prove that if $Q_0 < Q(u_{\rm c}(W_{\max}))$ and $\mathcal{D}$ is the smallest invariant domain for $\ronec$ containing $\Omega_{\rm f}$, then  $\mathcal{D} = \Omega_{\rm f} \cup \{u \in \Omega_{\rm c} \colon Q(u) \leq Q_0 \le p^{-1}(W_{\max}-v) \, v \}$.
\begin{enumerate}[label={``$\subseteq$''},itemindent=*,leftmargin=0pt]\setlength{\itemsep}{0cm}%
\item[``$\supseteq$'']
It suffices to observe that by assumption
\begin{align*}
&\mathcal{D}
\supseteq
\ronec[\Omega_{\rm f}'', \Omega_{\rm f}''](\R)
\supseteq
A \doteq \left\{u \in \Omega_{\rm c} \colon Q(u) = Q_0\right\}
~\Rightarrow\\
&\mathcal{D}
\supseteq
\ronec[\Omega_{\rm f}', A](\R)
\supseteq
B \doteq \left\{u \in \Omega_{\rm c} \colon W(u) = W_{\rm c} ,~ Q(u) \leq Q_0 \le p^{-1}(W_{\max}-v) \, v\right\}
~\Rightarrow\\
&\mathcal{D}
\supseteq
\ronec[\Omega_{\rm f}'', B](\R)
\supseteq
\left\{u \in \Omega_{\rm c} \colon Q(u) \leq Q_0 \le p^{-1}(W_{\max}-v) \, v \right\}.
\end{align*}
\item[``$\subseteq$'']
It suffices to observe that $\Omega_{\rm f} \cup \{u \in \Omega_{\rm c} \colon Q(u) \leq Q_0 \le p^{-1}(W_{\max}-v) \, v \}$ is an invariant domain for $\ronec$ by Definition~\ref{def:LWR-ARZ}.
\end{enumerate}

\item[\ref{inv-Rc-b}] 
We prove that if $Q_0 \ge Q(u_{\rm c}(W_{\max}))$ and $\mathcal{D}$ is the smallest invariant domain for $\ronec$ containing $\Omega_{\rm f}$, then  $\mathcal{D} = \Omega_{\rm f} \cup \{ u \in \Omega_{\rm c} \colon v = V_{\rm c} \}$.
\begin{enumerate}[label={``$\subseteq$''},itemindent=*,leftmargin=0pt]\setlength{\itemsep}{0cm}%
\item[``$\supseteq$'']
It suffices to observe that by assumption
\begin{align*}
&\mathcal{D}
\supseteq
\ronec[\Omega_{\rm f}'', \Omega_{\rm f}''](\R)
\supseteq
A \doteq \left\{u \in \Omega_{\rm c} \colon v = V_{\rm c} ,~ Q(u_{\rm f}(W(u))) > Q_0\right\}
~\Rightarrow\\
&\mathcal{D}
\supseteq
\ronec[\Omega_{\rm f}, A](\R)
\supseteq
\left\{ u \in \Omega_{\rm c} \colon v = V_{\rm c} \right\}
=
\mathcal{D} \cap \Omega_{\rm c}.
\end{align*}
\item[``$\subseteq$'']
It suffices to observe that $\Omega_{\rm f} \cup \{ u \in \Omega_{\rm c} \colon v = V_{\rm c} \}$ is an invariant domain for $\ronec$ by Definition~\ref{def:LWR-ARZ}.
\end{enumerate}

\item[\ref{inv-Rc-c}]
We prove that if $Q_0 \geq Q(u_{\rm c}(W_{\rm c}))$ and $\mathcal{D}$ is the smallest invariant domain for $\ronec$ containing $\Omega_{\rm c}$, then $\mathcal{D} = \Omega_{\rm c}$.
For this, it suffices to prove that $\Omega_{\rm c}$ is in fact an invariant domain.
It is easy to see that $\ronec[\mathcal{C}^2](\R) = \mathcal{R}_{\rm ARZ}[\mathcal{C}^2](\R) \subseteq \Omega_{\rm c}$.
Moreover, for any $(u_\ell,u_r) \in \mathcal{N}^2$ we immediately have that $u_\ell, \hat{u}, u_r \in \Omega_{\rm c}$ and $Q_0 \geq Q(u_{\rm c}(W_{\rm c})) \ge p^{-1}(W_{\rm c}-v_r) \, v_r$.
Hence, by \ref{f} we have that also $\check{u} \in \Omega_{\rm c}$.
As a consequence, by definition $\ronec[u_\ell,u_r](\R) = \mathcal{R}_{\rm ARZ}[u_\ell,\hat{u}](\R) \cup \mathcal{R}_{\rm ARZ}[\check{u},u_r](\R) \subseteq \Omega_{\rm c}$.

\item[\ref{inv-Rc-d}]
We prove that if $Q_0 < Q(u_{\rm c}(W_{\rm c}))$ and $\mathcal{D}$ is the smallest invariant domain for $\ronec$ containing $\Omega_{\rm c}$, then $\mathcal{D} = \Omega_{\rm c} \cup \{ u \in \Omega'_{\rm f} \colon Q(u) = Q_0 \}$.
\begin{enumerate}[label={``$\subseteq$''},itemindent=*,leftmargin=0pt]\setlength{\itemsep}{0cm}%
\item[``$\supseteq$'']
It suffices to observe that $\ronec[\mathcal{C}^2](\R) = \mathcal{R}_{\rm ARZ}[\mathcal{C}^2](\R) \subseteq \Omega_{\rm c}$ and that if $(u_\ell,u_r) \in \mathcal{N}^2$, then either $\check{u} \in \Omega_{\rm c}$ or $\check{u} \in \{ u \in \Omega'_{\rm f} \colon Q(u) = Q_0 \}$, see \ref{g}.

\item[``$\subseteq$'']
Let $\check{u}$ denote the unique element of $\{ u \in \Omega'_{\rm f} \colon Q(u) = Q_0 \}$.
Then it suffices to observe that $\ronec[\check{u},\Omega_{\rm c}](\R) \cup \ronec[\Omega_{\rm c},\check{u}](\R) \subseteq \{ u \in \Omega'_{\rm f} \colon Q(u) = Q_0 \}$.
\end{enumerate}
\end{enumerate}

\section*{References}

\bibliographystyle{elsarticle-harv}
\bibliography{biblio}

%% else use the following coding to input the bibitems directly in the
%% TeX file.

%\begin{thebibliography}{00}

%% \bibitem{label}
%% Text of bibliographic item

%\bibitem{}

%\end{thebibliography}
\end{document}